\documentclass[12pt,a4paper,reqno]{amsart}
\usepackage{amstext,fullpage}
\newtheorem{theorem}{Theorem}
\newtheorem{lemma}[theorem]{Lemma}

\newtheorem{proposition}[theorem]{Proposition}

\newtheorem{example}{Example}

\newcommand{\dist}{{\rm dist}}

\newcommand{\diam}{{\rm diam}}
\newcommand{\1}{{\bf 1\!\!1}}
\newcommand{\D}{\Delta}
\newcommand{\w}{\omega}
\newcommand{\N}{{\rm I\!N}}
\newcommand{\R}{{\rm I\!R}}

\newcommand{\f}{\varphi}
\newcommand{\Int}{{\rm Int}}
\newcommand{\Cl}{{\rm Cl}}
\newcommand{\summ}{\sum\limits}
\newcommand{\prodd}{\prod\limits}

\title{Exact packing dimension of random self-similar sets.}
\author{Artemi Berlinkov}
\address{Department of Mathemtics, Bar-Ilan University,
Ramat Gan, 5290002, Israel}
\email{artem.berlinkov@gmail.com}
\keywords{Exact packing dimension, packing measure,
random strong open set condition, random fractal.}
\date{\today}
\thanks{Research was supported by ISF 
grant 396/15; Center for Absorption in Science,  Ministry of Immigrant 
Absorption, State of Israel}
\subjclass[2010]{Primary 28A78, 28A80; Secondary 60D05, 60J80.}

\begin{document}
\begin{abstract}
We extend the results previously published on exact packing dimensions
of random recursive constructions to include
constructions satisfying commonly occurring conditions. We remove the restrictive assumption
that the diameter reduction ratios between the offsping
are either constant almost surely, or 0.
\end{abstract}
\maketitle

{\noindent \bf 1. }An exact packing dimension for the set $K\subset\R^m$ is the gauge function $\f(t)$
(i.e. a non-decreasing function such that $\f(0+)=0$), giving the set positive and finite packing measure, i.e. such that $0<{\mathcal P}^\f(K)<\infty$. No additional assumptions (like doubling condition) are imposed on the gauge function. In this article we consider random sets $K$ obtained in the limit of a {\it random recursive construction}, or homogeneous random fractal.

The definition of packing measure arose in 1980's as a counterpart to the Hausdorff measure. The definitions
and properties of Hausdorff and packing measures can be found in the monograph by P. Mattila (\cite{M}).  The definition of random recursive construction is adduced below, it can also be found in \cite{Be},\cite{F},\cite{GMW},\cite{MW}.

Let $n\in\N,$ $\D=\{1,\dots,n\}$.
Denote by $\D^*=\mathop{\bigcup}\limits_{j=0}^\infty \D^j$
the set of all finite sequences of numbers in $\D$,
and by $\D^\N$ the set of all their infinite sequences.
The result of concatenation of two
finite sequences $\sigma$ and $\tau$ from $\D^*$ is denoted
by $\sigma*\tau.$ For a finite sequence $\sigma$, its length
will be denoted by $|\sigma|.$
For a sequence $\sigma$ of length at least $k,$
$\sigma|_k$ is a sequence consisting
of the first $k$ numbers in $\sigma.$

Suppose that $J$ is a compact subset of ${\R}^d$ such that
$J=\Cl(\Int(J)),$ without loss of generality its diameter equals one.
The construction is a probability space $(\Omega, \Sigma, P)$
with a collection of random subsets of 
${\R}^d$ -- $\{J_\sigma(\w)|\w\in\Omega, \sigma\in\D^*\}$,
so that the following conditions hold.
\renewcommand{\theenumi}{\roman{enumi}}
\begin{enumerate}
\item\label{cond0} $J_\emptyset(\w)=J$ for almost all $\w\in\Omega,$
\item\label{cond1} For all $\sigma\in\D^*$ the maps $\w\to J_\sigma(\w)$
are measurable with respect to $\Sigma$ and the topology generated by the Hausdorff
metric on the space of compact subsets,
\item\label{cond2} For all $\sigma\in\D^*$ and $\w\in\Omega$,
the sets $J_\sigma$, if non-empty, are geometrically similar to $J$,
\item\label{cond3} For almost every $\w\in\Omega$ and all $\sigma\in\N^*,$ 
$i\in\N,$ $J_{\sigma*i}$ is a proper subset of $J_\sigma$ provided  $J_\sigma \neq \emptyset,$
\item\label{cond4} The construction satisfies the random {\it open set condition}:
if $\sigma$ and $\tau$ are two distinct sequences of the same length, then 
$\Int(J_\sigma)\cap\Int(J_\tau)=\emptyset$ a.s. and, finally,
\item\label{cond5}
The random vectors ${\mathbf T}_\sigma=(T_{\sigma*1}, T_{\sigma*2}, \dots),$
$\sigma\in\N^*$, are conditionally i.i.d. given that $J_\sigma(\w)\ne\emptyset$,
where $T_{\sigma*i}(\w)$ equals the ratio of the diameter of $J_{\sigma*i}(\w)$ to the diameter
of $J_\sigma(\w)$ .
\end{enumerate}

By \cite[Proposition 1]{Be}, we can replace conditional i.i.d. vectors in condition \eqref{cond5} by
an i.i.d. sequence of random vectors ${\mathbf T}_\sigma$, such that for non-empty offspring the
vector components give the diameter reduction ratios between the offspring and the parent.
The object of study is the random set 
$$K(\w)=\mathop{\bigcap}\limits_{k=1}^\infty
\mathop{\bigcup}\limits_{\sigma\in\D^k}J_\sigma (\w).$$

Graf et al. in \cite{GMW} have found under certain conditions
the exact Hausdorff dimension of the limit set
so that the $\f$-Hausdorff measure of
$K(\w)$ is positive and finite almost surely given $K(\w)\ne\emptyset.$
In \cite{BM} Berlinkov and Mauldin have found
an upper bound on exact packing dimension.
In \cite{Be2} Berlinkov proved that under two additional assumptions
this upper bound is the best. These assumptions will be shown late, one of them states
that there exists $\delta>0$, such that $P(T_i=\delta|T_i\ne 0)=1$. In this article we
show that the same result holds without the last assumption, and this allows for more applications
(see the examples). 

{\noindent \bf 2. }Now we will introduce more notation and state the result more precisely. 
Let $\alpha$ denote the almost sure
Hausdorff dimension of the non-empty random limit set $K(\w)$. It is the root of equation (see, e.g. \cite{F})
\[
E\Big[\sum\limits_{i=1}^n T_i^\alpha\Big]=1.
\]
By a result from \cite{BM}, the packing and Minkowski dimensions also equal $\alpha$. 

For $\sigma\in\D^*$ let $l_\sigma(\w)=\diam(J_\sigma(\w))$, $K_\sigma(\w)$ is the limit set we obtain
if the code tree is pruned to start at $\sigma$, $K_\sigma(\w)\subset J_\sigma(\w)$. It has been
noted (see, e.g. \cite{GMW}) that for $\sigma\in\D^m$ the sums below form a martingale whose limits
is denoted by
\[
X_\sigma=\lim\limits_{n\to\infty}\sum\limits_{|\tau|=n}l_{\sigma*\tau}^\alpha.
\]
Further, following \cite{MW}, we define a random measure $\mu_\w$ on the clopen cylinder sets
$A(\sigma)=\{\tau\in\D^\N|\tau_{|\sigma|}=\sigma\}\subset\D^\N$ by 
$\mu_\w(A(\sigma))=l_\sigma^\alpha(\w) X_\sigma(\w)$
and then we extend it to a random Radon measure on $\D^\N$. Following \cite{GMW}
we define measure $Q$ on the space $\D^\N\times\Omega$. For a Borel subset $B\subset
\D^\N\times\Omega$ we put
\[
Q(B)=\int\mu_\w(B_\w)dP(\w),
\]
where $B_\w$ are the sections of the set $B$ with respect to the second component.
Also we can extend the random variables $l_\sigma$, $X_\sigma$ and $T_\sigma$ onto the space
$\D^\N\times\Omega$ by considering
\[
l_k(\eta,\w)=l_{\eta|_k}(\w),\ X_k(\eta,\w)=X_{\eta|_k}(\w),\ T_k(\eta,\w)=T_{\eta|_k}(\w).
\]
We  show that under the folowing assumption holds the following theorem.

\noindent{\bf Assumption 1}. 
There exist $p_0,\rho>0,$ $s_0\ge 0$ and a collection
of events $R_\sigma$ in the $\sigma$-algebra
generated by random vectors $(T_{\tau*1},\dots,T_{\tau*n}),$
$\sigma\prec\tau,$ $|\tau|<|\sigma|+s_0,$
such that $R_\sigma\cap\{K_\sigma\ne\emptyset\}\ne\emptyset,$
for every $w\in R_\sigma\cap\{K_\sigma\ne\emptyset\}$
there exists $x\in K_\sigma$ with
$\dist(x,\partial J_\sigma)\ge\rho l_\sigma,$ and
$\int\limits_{R_\sigma}\summ_{|\tau|=s_0}
\prodd_{i=1}^{s_0}T^\alpha_{\sigma*\tau|_i}dP=p_0.$

This assumption was proven to hold for random homogeneous self-similar sets \cite[Proposition 2]{Be2} (i.e. sets, for which not only the reduction ratios but also the similarity maps between parent and its offspring are i.i.d.).  
\begin{theorem}
Suppose that the construction satisfies assumption~1, then 
\item{1.} If $P(0<X\le a)\asymp a^\beta,\ a\to 0$ and $\f(t)=t^\alpha g(t)$ is 
a gauge function, then
$\int\limits_{0^+}\frac{g^{\beta+1}(s)}{s}ds=+\infty$ implies
$P({\mathcal P}^\f(K(w))=+\infty|K(w)\ne\emptyset)=1$.
\item{2.} If $-\log P(0<X\le a)\asymp a^{1/\beta},\ a\to 0$
for all $a\in(0,1),$ then for $\f(t)=t^\alpha g(t)=t^\alpha |\log |\log t||^\beta$,
$P({\mathcal P}^\f(K(w))>0|K(w)\ne\emptyset)=1$.
\end{theorem}
We call case 1 in the above theorem the case of polynomial decay, and case 2 - the case of exponential decay.

To get the same conclusion, Berlinkov \cite[Theorem 2]{Be2} used an additional condition
that for $B_k=\{l_k^\alpha X_{k+s_0}<C\f(l_k\rho)\}\cap R_k$ holds
$Q(\varlimsup_k B_k)=1$, where $R_k$ is the extension of events $R_\sigma$
onto the space $\D^\N\times\Omega$: $R_k=\{(\eta,\w)\in\D^\N\times\Omega|\w\in R_{\eta|_k }\}$.
It was proven that the equality $Q(\varlimsup_k B_k)=1$ holds under an additional assumption

\noindent{\bf Assumption 2}.
There exists $\delta>0$ such that $P(T_i =\delta|T_i\ne 0)=1$ for all $i.$

We show that this second assumption is not necessary by proving the following
\begin{proposition}\label{main}
Let $B_k=\{l_k^\alpha X_{k+s_0}<C\f(l_k\rho)\}\cap R_k$, then $Q(\varlimsup_k B_k)=1$ for all constants
$C$ small enough in polynomial case, and in exponential case for all constants $C$ for which the series
 $\summ_{k=1}^\infty Q(B_k)$ diverges.
\end{proposition}
Together with \cite[Theorem 6]{BM}, in which the exact packing dimension bound was proven from the other side, this allows us to tell the exact packing dimension in the case of exponential decay or absence of one in the case of polynomial decay. The results from the article of Q. Liu (\cite{Liu}) sometimes allow us to determine the exponent of decay, and whether it is polynomial or exponential. Current state of results about small ball probabilities of the martingale limit to the extent of the author's knowledge covers only a fraction of the cases we may have.

{\noindent\bf 3. } Let us have look at some examples. In the following example the rate of decay is polynomial and the reduction ratios are independent, which allows us to apply \cite[Theorem 4.1(iii)]{Liu}.
\begin{example}
A random Cantor set.
\end{example}
Choose two numbers independently at random with respect to the uniform distribution from $J_0=[0,1]$. Let $J_1$, $J_2$ be the rightmost and leftmost subintervals of the partition (cf. \cite[Example 6.3]{GMW}). Then
the vector $(T_1,T_2)$ has density 2 with respect to the Lebesgue measure on the triangle $0\le T_1,T_2\le 1$,
$T_1+T_2\le 1$. By \cite[Example 4.2]{MW}, the Hausdorff dimension $\alpha=(\sqrt{17}-3)/2.$
From geometrical considerations
\[
P(T_1<x)=P(T_2<x)=x-x^2/2,
\]
and thus
\[
P(T_1^\alpha<x)=P(T_2^\alpha<x)=(x-x^2/2)^{1/\alpha}.
\]

By \cite[Theorem 4.1]{Liu}, $P(X<x)\asymp x^{2/\alpha},\ x\to 0$. Hence for all gauge functions the corresponding packing measure will be either 0 or infinite, depending on the integral test (note that we do not
require the doubling condition, as noted in the beginning of this note), therefore the exact packing dimension function does not exist as opposed to the exact Hausdorff dimension function, which equals
$t^\alpha|\log |\log t||^{1-\alpha}$ (cf. \cite[Example 6.3]{GMW}).

The next example is complicated by the fact that the reduction ratios are not independent.
\begin{example}
The zero set of the Brownian bridge.
\end{example}
Let $B_t$ denote the one-dimensional Brownian motion starting at 0. Then $B^0_t=B_t-t B_1$ is the Brownian bridge $(0\le t \le 1)$. This set can be considered as a random self-similar set as pointed out in
\cite[Example 6.1]{GMW}. If we define $\tau_1=\inf\{t\le 1/2\vert B^0_t=0\}$,
$\tau_2=\sup\{t\ge 1/2\vert B^0_t=0\}$, and set $J_0=[0,1]$, $J_1=[0,\tau_1]$, $J_2=[\tau_2,1]$ and
continue by recursion, then the zero set of the Brownian bridge can be represented as a random self-similar set with
reduction ratios $T_1=\tau_1$ and $T_2=1-\tau_2$ which are identically distributed but dependent random variables with joint distribution density (cf. \cite[(6.12)]{GMW})
\[
f(u,v)=\frac{1}{2\pi}\1_{[0,1/2]\times[0,1/2]}(u,v)(vt)^{-1/2}(1-v-t)^{-3/2}
\]

The Hausdorff dimension of this set $\alpha=1/2$ a.s.
By Example 6.3 in \cite{GMW}, for all $0<\xi<1/2$ holds $E[1/\min(T_1^\xi,T_2^\xi|T_i>0)]<\infty$.
By \cite[(6.9)]{GMW} and from the article of Chung \cite[(2.6)]{C} we can deduce that
\[
P(T_1<x)=P(T_2<x)=\frac 2\pi \arcsin x \asymp \sqrt{x},\ x\to 0
\]
and therefore  $P(T_1^\alpha<x)=P(T_2^\alpha<x)\asymp x, x\to 0$. Thus by \cite[Theorem 2.4]{Liu}
we have $P(X<x)=O(x^2)$. For the lower probability estimate of $P(X<x)$ we can use \cite[Lemma 3.4(i)]{Liu}.
To get an asymptotic estimate, according to that lemma, it suffices to prove
that $P(T_1^\alpha+T_2^\alpha<x)\ge Cx^2$ for some constant $C>0$ when $x$ approaches 0.
\[
P(T_1^\alpha+T_2^\alpha<x)=E_P[\1_{\{T_1^\alpha+T_2^\alpha<x\}}]=\frac{1}{2\pi}
\int_0^{1/2}\int_0^{1/2}\1_{\sqrt{v}+\sqrt{t}<x}(vt)^{-1/2}(1-v-t)^{-3/2}dvdt
\]
After a change of variables $z=\sqrt{v}$, $y=\sqrt{t}$, we obtain
\[
P(T_1^\alpha+T_2^\alpha<x)=
\frac {2}{\pi}\int_0^{\sqrt{1/2}}\int_0^{\sqrt{1/2}}\1_{z+y<x}(1-z^2-y^2)^{-3/2}dzdy.
\]
We continue by going over to polar coordinates and noting that the triangle $0\le z+y<x$ contains
quarter of a ball with center 0 and radius $\sqrt{3}x/2$
\begin{multline*}
P(T_1^\alpha+T_2^\alpha<x)\ge \frac{2}{\pi}\int_0^{\pi/2}\int_0^{\sqrt{3}x/2}\frac{rdrd\phi}{(1-r^2)^{3/2}}
=(1-r^2)^{-1/2}\vert^{\sqrt{3}x/2}_{0}\\
=\Big( \frac{1}{\sqrt{1-3x^2/4}}-1\Big)\asymp \frac{3x^2}{8},\ x\to 0.
\end{multline*}
Thus we can state that there is no exact packing dimension for the zero set of the Brownian bridge. 
%

{\noindent \bf 4.} In order to prove Proposition~\ref{main} we will use the following result from the article of Ortega and Wschebor (\cite{OW}), which is an extension of Borel-Cantelli lemma.
\begin{theorem}\label{ebc}
Let $\{B_k\}_{k=1}^\infty$ be a sequence of events in a probability space
such that $\summ_{k=1}^\infty Q(B_k)=\infty,$ and 
\begin{equation}\label{bcl:cond}
\mathop{\varliminf}\limits_{k\to\infty}\frac{\summ_{1\le i<j\le k}\big(
Q(B_i B_j)-Q(B_i)Q(B_j)\big)}{\big(\sum_{i=1}^k Q(B_i)\big)^2}\le 0,
\end{equation}
then $Q\big(\varlimsup B_k\big)=1.$
\end{theorem}
Divergence of the series $\summ_{k=1}^\infty Q(B_k)$ for all $C>0$ in polynomial case, 
and for all $C>\rho^{-\alpha}e^{\alpha E_Q[|\log T_1|]}t_0^{-\beta}$,
where $t_0=\mathop{\underline{\lim}}\limits_{x\to 0}-x^{-1/\beta}\log P(0<X\le x)<\infty$,
in exponential case was established in \cite[Proposition 1]{Be2}.  Then
Proposition \ref{main} immediately follows by Theorem~\ref{ebc}
from divergence of the series and the following lemma

\begin{lemma} In polynomial case there exists $M>0$ such that for all constants $C$ small enough
\[
\sum\limits_{i=1}^{j-1}Q(B_i B_j)-Q(B_i)Q(B_j)\le MQ(B_j).
\]
In the exponential case there exists $M>0$ such that the above inequality holds for all values
of $C$ for which the series $\summ_{k=1}^\infty Q(B_k)$ diverges.
\end{lemma}
\begin{proof}
To estimate the expressions $Q(B_i B_j)-Q(B_i)Q(B_j)$, note that by
\cite[Lemma 6]{Be2} we have
\[
Q(B_i)=p_0Q(l_i^\alpha X_{i+s_0}\le C\f(\rho l_i)),
\]
and by \cite[Lemma 10]{Be2}, for  all $i$ and $j$ such that $i+s_0<j$, we have
\[
Q(B_i B_j)\le p_0Q(Y_{i,j}\le C\rho^\alpha g(l_i \rho))Q(B_j),
\]
where $Y_{i,j}=X_{i+s_0}-X_j \prodd_{k=i+s_0+1}^j T_k^\alpha$. Thus
\begin{multline}\label{covar}
Q(B_i B_j)-Q(B_i)Q(B_j)\le \\
p_0 Q(B_j)
Q\Big(C\rho^\alpha g(l_i \rho)<X_{i+s_0}\le X_j \prodd_{k=i+s_0+1}^j T_k^\alpha+C\rho^\alpha g(l_i \rho)\Big)
\end{multline}
Note that
\begin{equation}
M_1:=E_Q[T_1^\alpha]<1,
\end{equation}
and choose $\delta\in(M_1,1)$.
The last factor in inequality \eqref{covar} can be estimated as follows
\begin{multline}\label{part}
Q\Big(C\rho^\alpha g(l_i \rho)<X_{i+s_0}\le X_j \prodd_{k=i+s_0+1}^j T_k^\alpha+C\rho^\alpha g(l_i \rho)\Big)
\le\\
Q\Big(C\rho^\alpha g(l_i \rho)<X_{i+s_0}\le \delta^{j-s_0-i}+C\rho^\alpha g(l_i \rho\Big)
+Q\Big(X_j \prodd_{k=i+s_0+1}^j T_k^\alpha>\delta^{j-s_0-i}\Big),
\end{multline}
and the last term in the above inequality \eqref{part} can be bounded by Markov's inequality  
\begin{equation}\label{Markov}
Q\bigg(X_j \prodd_{k=i+s_0+1}^j T_k^\alpha>\delta^{j-s_0-i}\bigg)<E_Q[X_0](M_1/\delta)^{j-i-s_0},
\end{equation}
Note that by \cite[Corollary to Theorem 2.2]{Liu} we have the $Q$-density of $Q(0<X\le a)$ bounded by
constant $M_2$ on $[0,1]$, so 
\begin{multline}
\sum\limits_{i=1}^{j-1}Q\Big(C\rho^\alpha g(l_i \rho)<X_{i+s_0}\le X_j \prodd_{k=i+s_0+1}^j T_k^\alpha+C\rho^\alpha g(l_i \rho)\Big)
\le\\
s_0+\sum\limits_{i=1}^{j-s_0-1}\bigg[M_2 \delta^{j-s_0-i} +E[X_0](M_1/\delta)^{j-i-s_0}
+Q(C\rho^\alpha g(l_i \rho)>1- \delta^{j-s_0-i})\bigg]
\end{multline}
Note that by \cite[(1.14)]{MW}, 
$l_i(\w)\searrow0$ a.s., and suppose that $\varlimsup_{0+} g(t)=\theta>0$ (this can only happen in
case of polynomial decay). There exists $t_0>0$ such that for all $t<t_0$ we have $g(\rho t)<2\theta$.
Then for all $C$ such that $(1-\delta)/C\rho^\alpha>2\theta$ we have
\[
\sum\limits_{i=1}^\infty Q(C\rho^\alpha g(l_i \rho)>1- \delta) 
\le\sum\limits_{i=1}^\infty Q(l_i>t_0)\le
\sum\limits_{i=1}^\infty E_Q[T_1]^i/t_0<\infty.
\]
In case $\varlimsup_{0+}g(t)=0$ the considerations the inequluality
$\sum\limits_{i=1}^\infty Q(C\rho^\alpha g(l_i \rho)>1- \delta)<\infty$ holds
for all $C$ by similar considerations.
\end{proof}

\end{document}